\documentclass[12pt,a4paper]{amsart}
\usepackage[left=2.4cm,right=2.4cm,top=3.0cm,bottom=2.8cm, footskip=18pt]{geometry}
\usepackage{amsmath,amssymb,amsfonts,amscd,amsthm,wasysym,xcolor,enumitem,indentfirst,graphicx,booktabs,mathtools}
\usepackage[bookmarksnumbered,colorlinks,linktocpage,pagebackref,plainpages]{hyperref}
\hypersetup{pdfstartview={FitH}}

\newtheorem{thmx}{Theorem}

\numberwithin{equation}{section}
\newtheorem{theorem}{Theorem}[section]

\newtheorem{lemma}[theorem]{Lemma}

\theoremstyle{definition}

\newtheorem*{remark}{Remark}

\newtheorem{Qx}[thmx]{Question}

\newcommand{\ddc}{{\rm dd^c}}
\newcommand{\dd}{\mathrm d}

\DeclareMathOperator{\rank}{rank}

\allowdisplaybreaks[4]
\renewcommand{\arraystretch}{1.2}

\makeatletter
\@namedef{subjclassname@2020}{\textup{2020} Mathematics Subject Classification}
\makeatother

\begin{document}
\title[Pluricomplex Poisson kernel and complex Monge--Amp\`{e}re equations]{On the pluricomplex Poisson kernel and the associated complex Monge--Amp\`{e}re equation}
\date{\today}

\author{Xieping Wang}
\address{CAS Wu Wen-Tsun Key Laboratory of Mathematics and School of Mathematical Sciences, University of Science and Technology of China, Hefei 230026, Anhui, People's Republic of China}
\email{xpwang008@ustc.edu.cn}

\thanks{The author was partially supported by the NSFC (Grant No. 12371083) and the Fundamental Research Funds for the Central Universities (Grant Nos. WK0010000099 and WK3470000030).}


\dedicatory{To Li and Rui'an}

\begin{abstract}
We give a geometric characterization of the pluricomplex Poisson kernel of bounded strongly linearly convex domains in $\mathbb C^{n+1}$ in terms of their foliation by all complex geodesic discs whose closure  contains a fixed boundary point, thus confirming a conjecture posed by Bracci et al. in 2009, even in  greater generality. The proof relies crucially on our previous results obtained in a series of two papers. We also revisit the homogeneous complex Monge--Amp\`{e}re equation associated with the pluricomplex Poisson kernel and show that its solutions are far from unique in general by constructing a family of pairwise nonproportional continuous solutions on the unit ball in $\mathbb C^{n+1}$. This sharply contrasts with the uniqueness of solutions to  the corresponding equation for the pluricomplex Green function.
\end{abstract}

\maketitle

\section{Introduction}
As is well known, the Green function and the Poisson kernel are two fundamental objects in classical potential theory. In the planar case, they transform naturally under conformal mappings and hence play an important role in complex analysis. In higher complex dimensions, their counterparts in pluripotential theory are the pluricomplex Green
function and the pluricomplex Poisson kernel.

To be specific, let $\Omega$ be a bounded domain in $\mathbb C^{n+1}$ and let  ${\rm Psh}(\Omega)$ denote the set of all plurisubharmonic functions on $\Omega$. Then the pluricomplex Green function $g_{\Omega}(\,\cdot\,,w)$ of $\Omega$ with pole at $w\in\Omega$ is defined to be
\begin{equation*}\label{Greenkernel}
   g_{\Omega}(\,\cdot\,,w) \coloneqq \sup\Big\{u\in {\rm Psh}(\Omega)\!: u<0,
   \;\limsup\limits_{z\to w}(u(z)-\log|z-w|)<\infty \Big\}.
\end{equation*}
When $\Omega$ is hyperconvex, i.e. there exists a negative continuous function $u\in {\rm Psh}(\Omega)$ such that $\{u<t\}\subset\subset \Omega$ for all $t<0$, a classical result of Demailly \cite{Demailly87} states that for every $w\in\Omega$, $g_{\Omega}(\,\cdot\,,w)$ is continuous and is the unique solution to the homogeneous complex Monge--Amp\`{e}re equation
\begin{equation}\label{eq:Inter-MA}
  \begin{cases}
   u\in {\rm Psh}(\Omega)\cap
   L_{{\rm loc}}^{\infty}(\Omega\!\setminus\!\{w\}), \\
   (\ddc u)^{n+1}=0        &{\rm on}\,\,\ \Omega \!\setminus\! \{w\}, \\
   \lim_{z\to x}u(z)=0     &{\rm for}\,\  x\in\partial\Omega, \\
   u(z)-\log|z-w|=O(1)     &{\rm as} \, \ \ z\to w.\\
  \end{cases}
\end{equation}
Actually, this equation was first investigated by Lempert \cite{Lempert81, Lempert83, Lempert84} on a restricted class of bounded strongly linearly convex domains $\Omega\subset\mathbb C^{n+1}$ with sufficiently smooth boundary, using his theory of complex geodesics. In particular, he discovered a one-to-one correspondence between the pluricomplex Green function $g_{\Omega}(\,\cdot\,,w)$ and the singular foliation of the underlying domain $\Omega$ by all complex geodesic discs passing through $w$.
We refer the reader to the surveys \cite{Bracci_survey10, Blocki-survey14, Chen-Survey15} for more results concerning the pluricomplex Green function and its profound applications, and to \cite{Bracci-JEMS10, Chen-APDE17} for recent developments on related topics.


In 2005, Bracci, Patrizio and Trapani \cite{Bracci-MathAnn05, Bracci-Trans09} initiated a boundary analogue of the above theory, building on the foundational work of Lempert \cite{Lempert81, Lempert84} (see also \cite{Abatebook, KW13}), Chang--Hu--Lee \cite{Chang--Hu--Lee88} and Huang \cite{Huang-Illinois94, Huang-Pisa94}; see also \cite{Huang-Wang22, Bracci-Adv21, Wang24, ABF25} for recent developments in this direction. To be more specific, let $\Omega\subset\mathbb C^{n+1}$ be a bounded \textit{strongly linearly convex}\footnote{\,The precise definition plays no role in the main results of this paper, so we do not explicitly introduce it here but refer the interested reader to, e.g., \cite{Andersson04, Huang-Wang22} for details.} domain with $C^3$-boundary and let $\langle\;,\,\rangle$ denote the standard Hermitian inner product on $\mathbb C^{n+1}$.
Then Bracci et al. \cite{Bracci-MathAnn05, Bracci-Trans09} (see also \cite{Huang-Wang22, Bracci-Adv21}) defined the pluricomplex Poisson kernel $P_{\Omega,\, p}$ of $\Omega$ with singularity at $p\in\partial\Omega$ as follows:
\begin{equation*}\label{Poissonkernel}
  \begin{split}
   P_{\Omega,\, p}
   \coloneqq \sup\Big\{&u\in {\rm Psh}(\Omega)\!: \limsup_{z\to x} u(z)\leq 0 \ \hbox{for all}
   \ x\in \partial \Omega\!\setminus\!\{p\},\\
   &\liminf_{t\to 1^{-}} |u(\gamma(t))(1-t)|
   \geq 2\,{\rm Re}\, \langle\gamma'(1),\, \nu_p\rangle^{-1} \ \hbox{for all}\ \gamma\in \Gamma_p
   \Big\},
  \end{split}
\end{equation*}
where $\Gamma_p$ is the set of all non-tangential $C^1$ curves $\gamma\!:[0, 1]\to \Omega\cup\{p\}$ terminating at $p$ and with $\gamma([0,1))\subset\Omega$, and $\nu_p$ denotes the unit outward normal to $\partial\Omega$ at $p$. Also they (and Huang--Wang  \cite{Huang-Wang22}) proved that $P_{\Omega,\, p}$ is a continuous solution to the following homogeneous complex Monge--Amp\`{e}re equation on $\Omega$ with $p\in\partial\Omega$, which is a boundary analogue of equation \eqref{eq:Inter-MA}:
\begin{equation}\label{eq:Bdry-MA}
  \begin{cases}
   u\in {\rm Psh}(\Omega)\cap L_{{\rm loc}}^{\infty}(\Omega), \\
   (\ddc u)^{n+1}=0         &{\rm on} \,\,\ \Omega, \\
   \lim_{z\to x}u(z)=0      &{\rm for} \,\  x\in\partial\Omega\!\setminus\!\{p\}, \\
   u(z)\approx -|z-p|^{-1}  &{\rm as}  \, \ \ z\to p\, \ {\rm nontangentially},
  \end{cases}
\end{equation}
where the last condition means that for every $\alpha>1$, there exists a constant $C_{\alpha}>1$ such that
   $$
   C_{\alpha}^{-1} <-u(z)|z-p|<C_{\alpha}
   $$
holds for all $z\in \Gamma_{\alpha}(p)$ sufficiently close to $p$, and
   $$
   \Gamma_{\alpha}(p)=\big\{z\in\Omega\!: |z-p|<\alpha\, {\rm dist}(z,\, \partial\Omega)\big\}.
   $$
is the non-tangential approach region in $\Omega$ with vertex $p$ and aperture $\alpha$. Here, ${\rm dist}(\,\cdot\,,\partial\Omega)$ denotes the Euclidean distance to the boundary $\partial\Omega$.

Indeed, Bracci, Patrizio and Trapani proved the preceding result by making essential use of the foliation\footnote{\,The existence of such a foliation for strongly linearly convex domains with $C^3$-smooth boundary was established in \cite{Huang-Wang22}; see also \cite{Chang--Hu--Lee88} for earlier partial results.} of the domain $\Omega$ by all complex geodesic discs whose closure contains the boundary point $p$. Roughly speaking, $P_{\Omega,\, p}$ can be reconstructed by gluing together in a continuous manner the Poisson kernels on the individual leaves of this foliation. This naturally raises the following two questions:

\begin{Qx}\label{Q:uniqueness1}
Is the pluricomplex Poisson kernel $P_{\Omega,\, p}$, up to a positive constant multiple, the only plurisubharmonic function that can arise from such a gluing construction?
\end{Qx}

\begin{Qx}\label{Q:uniqueness2}
Is the pluricomplex Poisson kernel $P_{\Omega,\, p}$, up to a positive constant multiple, the unique solution to equation \eqref{eq:Bdry-MA}?
\end{Qx}


When $\Omega$ is strongly convex with $C^\infty$-boundary, Bracci, Patrizio, and Trapani \cite{Bracci-Trans09} proved a related result that gives an affirmative answer to Question \ref{Q:uniqueness1} under the additional assumption that $u$ is $C^2$-smooth. They further conjectured that this regularity assumption could be removed. The following theorem gives an affirmative answer to Question \ref{Q:uniqueness1}, thus confirming their conjecture even in the more general setting of strongly linearly convex domains. It is also worth noting that, even in the $C^2$-smooth case, our proof here seems to be simpler than theirs (see Section \ref{sect:Pf-MainThm1} for comparison and further explanation).

\begin{theorem}\label{thm:main1}
Let $\Omega\subset\mathbb C^{n+1}$ be a bounded strongly linearly convex domain with $C^{5,\,\alpha}$-boundary, where $0<\alpha<1$, and let $p\in\partial\Omega$. Suppose $u\in {\rm Psh}(\Omega)$ satisfies that $\lim_{z\to x}u(z)=0$ for all $x\in\partial\Omega \!\setminus\! \{p\}$ and that its restriction to each complex geodesic disc of $\Omega$ whose closure contains $p$ is harmonic. Then there exists $c\ge 0$ such that
  $$
  u=cP_{\Omega,\, p}.
  $$
\end{theorem}

Roughly speaking, this result provides a geometric characterization of the pluricomplex Poisson kernel $P_{\Omega,\, p}$ in terms of the foliation of the domain $\Omega$ by all complex geodesic discs whose closure contains $p$. We refer the reader to \cite{Bracci-Trans09} for other characterizations of the kernel, and also to \cite{Bracci-JEMS10, Bracci-MathAnn05} for its important applications in the theory of holomorphic semigroups and other areas. It turns out that the kernel function $P_{\Omega,\, p}$ bears a strong resemblance to the classical Poisson kernel in real potential theory, much as the pluricomplex Green function $g_{\Omega}(\,\cdot\,,w)$ does to the classical Green function.

As for Question \ref{Q:uniqueness2}, which was first posed by Huang and the author in \cite{Huang-Wang22} and seems to be very interesting from the PDE point
of view, the answer turns out to be negative even in the simplest case of the unit ball $\mathbb B^{n+1}\subset\mathbb C^{n+1}$, in sharp contrast with the uniqueness phenomenon for equation \eqref{eq:Inter-MA}. To illustrate the abundance of solutions to equation \eqref{eq:Bdry-MA} on $\mathbb B^{n+1}$ for every positive integer $n$, we establish the following result.

\begin{theorem}\label{thm:main2}
There is a pairwise distinct family $\{u_{\lambda}\}_{\lambda\in(0,\, 1)}$ of continuous plurisubharmonic functions on $\mathbb B^{n+1}$ with the following
properties:
\begin{enumerate}[label=\textup{(\roman*)}, leftmargin=2.0pc, parsep=4pt]
   \item  $(\ddc u_{\lambda})^{n+1}=0$ on $\mathbb B^{n+1}$ and
       $$
       P_{\mathbb B^{n+1},\, e_1} \leq u_{\lambda} \leq e^{-\lambda} P_{\mathbb B^{n+1},\, e_1};
       $$

    \item Each $u_{\lambda}$ is not a constant multiple of $P_{\mathbb B^{n+1},\, e_1}$,
\end{enumerate}
where $e_{1}\coloneqq (1, 0,\ldots, 0)\in\partial\mathbb B^{n+1}$.
\end{theorem}

The construction of such a family $\{u_{\lambda}\}_{\lambda\in(0,\, 1)}$ is motivated by the well-known relationship between convex functions and
plurisubharmonic functions (see, e.g., \cite[Chapter I]{Demaillybook}),  as well as by the optimization theory.

Note that Theorem \ref{thm:main2} also shows that the answer to \cite[Question 7.6]{Bracci-Trans09} is {\it negative}, in general.

The paper is organized as follows. In Section~\ref{sect:Pf-MainThm1} we prove Theorem~\ref{thm:main1}. To highlight the main idea before
introducing the additional technical details, we first treat the
$C^2$-smooth case and then turn to the general case.
Section~\ref{sect:Pf-MainThm2} is devoted to the construction of a family of continuous plurisubharmonic functions on  $\mathbb B^{n+1}$ with the desired properties in Theorem~\ref{thm:main2}.

\section{Proof of Theorem~\ref{thm:main1}}\label{sect:Pf-MainThm1}
We first prove the theorem in the $C^2$-smooth case to highlight the main idea and, for the interested reader, to facilitate comparison with the original proof in \cite{Bracci-Trans09}.  Our proof in this case relies crucially on the recent results obtained in \cite{Huang-Wang22, Wang24} and avoids using the
Jacobi-field representation formula established in \cite{S-Trapani_JGA02}, which played a key role in the original proof. We then treat the general case by introducing the additional technique required.

\subsection{The $C^2$-smooth case}
Note that $P_{\Omega,\,p}\in C^2(\Omega)$ when $\partial\Omega$ is $C^{5,\,\alpha}$-smooth, in view of \cite[Theorem D]{Wang24}.

\begin{proof}[Proof of Theorem \ref{thm:main1} in the case when $u$ is $C^2$-smooth]
Arguing as in \cite{Bracci-Trans09}, we know that either $u\equiv0$, in which case the conclusion holds with $c=0$, or $u<0$ on $\Omega$. We henceforth assume that the latter alternative holds. Then by assumption, for each complex geodesic $\varphi$ of $\Omega$ with $\varphi(1)=p$, $u\circ\varphi$ is a negative harmonic function on the unit disc $\Delta\subset\mathbb C$ with boundary value zero on $\partial\Delta \!\setminus\! \{1\}$. Hence by the classical Herglotz representation theorem, $u\circ\varphi$ coincides, up to a negative constant multiple, with the Poisson kernel $P$ on $\Delta$ (with singularity at $1$):
  $$
  P(\zeta)=\frac{1-|\zeta|^2}{|1-\zeta|^2}.
  $$
The same property holds also for $P_{\Omega,\, p}$. Consequently, the quotient $u/P_{\Omega,\, p}$ is constant on each complex geodesic disc of $\Omega$ whose closure contains $p$. It remains to prove that this constant is actually independent of these geodesic discs.

To prove this we argue as follows. Up to a unitary transformation of $\mathbb C^{n+1}$, we may assume that the unit outward normal to $\partial\Omega$ at $p$ is $(1, 0,\ldots, 0)$. Then according to
\cite[Theorem 1.1]{Huang-Wang22}, for every $v\in\mathbb C^n$ there exists a unique complex geodesic $\varphi_v$ of $\Omega$ such that $\varphi_v(1)=p$, $\varphi_v'(1)=(1,\, v)/(1+|v|^2)$ and
\[
 \left.\frac{\dd}{\dd\theta}\right|_{\theta=0} \bigl|\varphi^{\ast}_v(e^{i\theta})\bigr|=0,
\]
where $\varphi^{\ast}_v$ is the dual mapping of $\varphi_{v}$. Define $\psi_v\coloneqq \varphi_v\circ\sigma_v$, where $\sigma_v$ is the unique element of ${\rm Aut}(\Delta)$ satisfying
\[
 \frac{1+\sigma_v(\zeta)}{1-\sigma_v(\zeta)}=\frac{1}{1+|v|^2}\frac{1+\zeta}{1-\zeta}, \quad \zeta\in\Delta.
\]
In this way, each $\psi_v$ satisfies the normalization:
\[
 \psi_v(1)=p,\qquad \psi_v'(1)=(1,\, v)\qquad \text{and} \qquad P_{\Omega,\, p}\circ\psi_v=-P.
\]
Consequently there exists a function $\lambda\!:\mathbb C^n\to(0,\, \infty)$ such that
\[
 u\circ\psi_v=-\lambda(v)P \quad {\rm on }\ \, \Delta
\]
for all $v\in\mathbb C^n$, and the problem reduces to showing that $\lambda$ is constant.

By \cite[Theorem A]{Wang24}, the mapping
\[
 \mathbb C^n\ni v\mapsto \psi_v\in \mathcal O(\Delta,\, \mathbb C^{n+1})\cap C^1(\overline\Delta,\,\mathbb C^{n+1})
\]
is $C^1$-smooth. Thus $\lambda(v)=-u\circ\psi_v(0)\in C^1(\mathbb C^n)$. Now fix an arbitrary point $v_0\in\mathbb C^n$. To prove that $(\dd\lambda)_{v_0}\equiv 0$, consider the variation field of $\{\psi_v\}_{v\in\mathbb C^n}$ along the complex geodesic $\psi_{v_0}$:
\[
 J_v \coloneqq \left.\frac{\dd}{\dd t}\right|_{t=0}\psi_{v_0+tv}\in\mathcal O(\Delta,\,\mathbb C^{n+1})\cap C^1(\overline\Delta, \,\mathbb C^{n+1}),
\]
which satisfies
\[
 J_v(1)=0  \qquad \text{and} \qquad  J_v'(1)=\left.\frac{\dd}{\dd t}\right|_{t=0}\psi_{v_0+tv}'(1)=(0,\, v)
\]
for all $v\in\mathbb C^n$, in view of the normalization of $\{\psi_v\}_{v\in\mathbb C^n}$. Hence
\begin{equation}\label{eq:asy-Jacobi}
 \lim_{\zeta\to1}\frac{J_v(\zeta)}{\zeta-1}=(0,\, v), \quad v\in\mathbb C^n.
\end{equation}
To proceed further, note that the mapping
\[
 \frac{\partial u}{\partial z}\circ\psi_{v_0} \coloneqq \left(\frac{\partial u}{\partial z_1},\ldots,
 \frac{\partial u}{\partial z_{n+1}}\right)\circ\psi_{v_0}
\]
is holomorphic on $\Delta$. Indeed, the harmonicity of $u\circ\psi_{v_0}$ implies
\[
 \sum_{j,\, k=1}^{n+1}
 \frac{\partial^2u}{\partial z_j\partial\overline z_k}\circ\psi_{v_0} (\psi_{v_0}')_j\, \overline{(\psi_{v_0}')}_k
 =\frac{\partial^2(u\circ\psi_{v_0})}{\partial\zeta\,\partial\overline\zeta}=0 \quad {\rm on }\ \, \Delta,
\]
which, together with the Cauchy--Schwarz inequality applied to the complex Hessian of $u$, in turn forces
\[
 \sum_{k=1}^{n+1}
 \frac{\partial^2u}{\partial z_j\partial\overline z_k}\circ\psi_{v_0} \overline{(\psi_{v_0}')}_k=0 \quad {\rm on }\ \, \Delta
\]
for $j=1,\ldots,n+1$. By the chain rule, this means exactly that each component $\frac{\partial u}{\partial z_j}\circ\psi_{v_0}$ is holomorphic on $\Delta$, as desired.

Now consider the holomorphic function
\[
 h_v:=\left\langle J_v, \, \overline{\frac{\partial u}{\partial z}\circ\psi_{v_0}}\right\rangle
 =\sum_{j=1}^{n+1} (J_v)_j\,\frac{\partial u}{\partial z_j}\circ\psi_{v_0}.
\]
Differentiating the identity $u\circ\psi_{v_0+tv}=-\lambda(v_0+tv)P$ with respect to $t$ and evaluating at $t=0$ yield
\[
 2{\rm Re}\,h_v=-(\dd\lambda)_{v_0}(v)P.
\]
Then there exists $c_v\in\mathbb R$ such that
\[
 h_v(\zeta)=-\frac{(\dd\lambda)_{v_0}(v)}{2}\frac{1+\zeta}{1-\zeta}+ic_v, \quad \zeta\in\Delta.
\]
In particular,
\begin{equation}\label{eq:asy-pairing}
 \lim_{\zeta\to 1} \left\langle \frac{J_v(\zeta)}{\zeta-1}, \, \overline{(\zeta-1)^2\frac{\partial u}{\partial z}\circ\psi_{v_0}(\zeta)} \right\rangle
 =\lim_{\zeta\to1}(\zeta-1)h_v(\zeta)=(\dd\lambda)_{v_0}(v), \quad v\in\mathbb C^n.
\end{equation}
On the other hand, differentiating the identity $u\circ\psi_{v_0}=-\lambda(v_0)P$ on $\Delta$ yields
\[
 \left\langle
 \psi_{v_0}'(\zeta), \, \overline{(\zeta-1)^2\frac{\partial u}{\partial z}\circ\psi_{v_0}(\zeta)}
 \right\rangle
 =-\lambda(v_0)
\]
for all $\zeta\in\Delta$. Combining these two equalities with \eqref{eq:asy-Jacobi} and the fact that $\psi_{v_0}'(1)=(1,\, v_0)$, we conclude that $(\zeta-1)^2\frac{\partial u}{\partial z}\circ\psi_{v_0}(\zeta)$ converges to some $w_0\in\mathbb C^{n+1}$ as $\zeta\to1$. Then it follows from \eqref{eq:asy-Jacobi} and \eqref{eq:asy-pairing} that
\[
 (\dd\lambda)_{v_0}(v)=\big\langle(0,\,v), \, \overline{w}_0\big\rangle, \quad v\in\mathbb C^n.
\]
Since $(\dd\lambda)_{v_0}$ takes values in $\mathbb R$, this identity implies $(\dd\lambda)_{v_0}\equiv 0$, as desired. The proof is complete.
\end{proof}

\subsection{The general case}
We borrow some basic idea from the viscosity technique of solving degenerate complex Monge--Amp\`{e}re equations (see \cite{EGZ-CPAM11, GZ-MAeqnbook}). We first prove
the following simple fact.

\begin{lemma}\label{lem:viscosity}
Let $\Omega\subset\mathbb{C}^n$ be a neighborhood of the origin $0$ and let $u\in{\rm Psh}(\Omega)$. If $v\in C^2(\Omega)$ touches $u$ from above at $0$, then its complex Hessian at $0$ is necessarily positive semidefinite.
\end{lemma}

\begin{proof}
By assumption, for every $w\in\mathbb{C}^n \!\setminus\! \{0\}$ the function $\zeta\mapsto u(\zeta w)$ is subharmonic near $0$ and is touched from above
at $0$ by the $C^2$ function $\zeta\mapsto v(\zeta w)$. Together with the submean value property, this implies
\[
 v(0)=u(0) \le \frac{1}{2\pi}\int_0^{2\pi}u(re^{i\theta}w)\,\dd\theta
 \le \frac{1}{2\pi}\int_0^{2\pi}v(re^{i\theta}w)\,\dd\theta
\]
for all sufficiently small $r>0$. Consequently,
\[
\sum_{j,\, k=1}^{n} \frac{\partial^2 v}{\partial z_j\partial\overline{z}_k}(0) w_j\overline{w}_k
=\left.\bigg(\frac{\partial^2}{\partial\zeta\partial\overline{\zeta}} v(\zeta w)\bigg)\right|_{\zeta=0}
=\lim_{r\to 0}\frac{1}{r^2}\left(\frac{1}{2\pi}\int_0^{2\pi}v(re^{i\theta}w)\,\dd\theta-v(0)\right)
\ge 0
\]
for all $w\in\mathbb{C}^n \!\setminus\!\{0\}$. The desired conclusion follows.
\end{proof}

We also need the following calculus lemma.

\begin{lemma}\label{lem:calculus}
Let $D$ be a domain in $\mathbb{R}^n$ and $f\!:D\to\mathbb{R}$ be an upper semicontinuous function with the property that for every $x\in D$, each $C^2$ function on a neighborhood of $x$ that touches $f$ from above at $x$ has vanishing differential at $x$. Then $f$ is constant.
\end{lemma}

\begin{proof}
It suffices to show that for every $x\in D$ and $R>0$ with $B(x, R)\subset\subset D$,
\[
\sup_{B(x,\, R)} f=f(x).
\]
Once this is established, it follows that $f$ is locally constant on $D$, and hence constant by the connectedness of $D$.

We argue by contradiction and suppose that there exists $x_0\in D$ and $R_0>0$ with $B(x_0, R_0)\subset\subset D$ such that
\[
\sup_{B(x_0,\, R_0)} f>f(x_0).
\]
Choose $y_0\in B(x_0, R_0)$ such that $f(y_0)>f(x_0)$, and consider the auxiliary function
\[
g(x)\coloneqq \frac{1}{R_0^2-|x-x_0|^2}, \quad x\in B(x_0,  R_0).
\]
Then there exists a sufficiently small $\varepsilon_0>0$ such that
$$f(y_0)-\varepsilon_0 g(y_0)>f(x_0).$$ Since $f$ is bounded above on $B(x_0,R_0)$ and $g(x)$ tends to $+\infty$ as
$x\to\partial B(x_0,R_0)$, it follows that $f-\varepsilon_0 g$ attains its maximum on $B(x_0, R_0)$ at some point $\widetilde{x}_0\ne x_0$. Consequently
$\varepsilon_0\bigl(g-g(\widetilde{x}_0)\bigr)+f(\widetilde{x}_0)$ touches $f$ from above at $\widetilde{x}_0$, hence has vanishing
differential there by assumption. But this contradicts the
fact that $\widetilde{x}_0\ne x_0$, and the claim follows.
\end{proof}

With Lemmas \ref{lem:viscosity} and \ref{lem:calculus} at disposal, we now indicate how to adapt the argument for the $C^2$-smooth case in the preceding subsection to the general case.

\begin{proof}[Proof of Theorem \ref{thm:main1} in the general case]
We retain the notation introduced in the preceding proof and consider the mapping
\[
\Phi\!:\mathbb{C}^n\times\Delta\rightarrow\Omega, \quad (v,\, \zeta) \mapsto \psi_v(\zeta).
\]
Then $\Phi$ is a $C^2$ diffeomorphism. To see this, let $\Psi=\Psi_p$ denote the boundary spherical representation for $\Omega$ introduced in \cite{Huang-Wang22}, which is a  $C^2$ diffeomorphism between $\Omega$ and $\mathbb{B}^{n+1}$, in view of \cite[Theorem C]{Wang24}; more explicitly, in the present context $\Psi$ is determined by the requirement
\[
\Psi\circ\psi_v=\eta_v\circ\sigma_v, \quad v\in\mathbb{C}^n,
\]
where
\[
\eta_v(\zeta)=e_1+(\zeta-1)\frac{(1,\, v)}{1+|v|^2}.
\]
Then an explicit formula shows that
\[
\Psi\circ\Phi(v,\, \zeta)=\eta_v\circ\sigma_v(\zeta)
\]
is a $C^\infty$ diffeomorphism between $\mathbb{C}^n\times\Delta$ and $\mathbb{B}^{n+1}$, hence $\Phi=\Psi^{-1}\circ(\Psi\circ\Phi)$ is a $C^2$ diffeomorphism, as desired.
On the other hand, note that the function $\lambda\!:\mathbb{C}^n\to (0,\, \infty)$ determined by the relationship that $u\circ\psi_v=-\lambda(v)P$ on $\Delta$ is merely lower semicontinuous. Indeed, since $u$ is upper semicontinuous and the mapping $\mathbb{C}^n\ni v\mapsto\psi_v(0)\in \Omega$ is continuous, it follows that $\lambda(v)=-u\circ\psi_v(0)$ is lower semicontinuous.

To apply Lemma \ref{lem:calculus} to conclude that $\lambda$ is constant, we fix an arbitrary point $v_0\in\mathbb{C}^n$ and suppose $\widetilde\lambda$ is a $C^2$ function on a neighborhood $V_0$ of $v_0$ that touches $\lambda$ from below at $v_0$. We then need to show $(\dd\widetilde\lambda)_{v_0}\equiv 0$. To this end, define
\[
\widetilde u \coloneqq -\big((\widetilde\lambda\circ\pi_1)(P\circ\pi_2)\big)\circ(\left.\Phi\right|_{V_0\times\Delta})^{-1},
\]
where $\pi_1$ and $\pi_2$ denote the projections of $\mathbb{C}^n\times\Delta$ to $\mathbb{C}^n$ and $\Delta$, respectively. Then $\widetilde u$ is a $C^2$ function on $\Phi(V_0\times\Delta)\subset\Omega$ that
\begin{itemize}[leftmargin=1.5pc, parsep=4pt]
  \item satisfies $\widetilde u\circ\psi_v=-\widetilde\lambda(v)P$ on $\Delta$ for all $v\in V_0$, and
  \item touches $u$ from above on the whole geodesic disc $\psi_{v_0}(\Delta)=\Phi(\{v_0\}\times\Delta)$.
\end{itemize}
Thus by Lemma \ref{lem:viscosity}, the complex Hessian of $\widetilde u$ at each point of $\psi_{v_0}(\Delta)$ is positive semidefinite. Now the same argument as in the latter half of the third paragraph of the preceding proof shows that all the functions
\[
\frac{\partial\widetilde u}{\partial z_j}\circ\psi_{v_0}, \quad j=1,\ldots,n+1,
\]
are holomorphic on $\Delta$. This in turn enables us to define a new holomorphic function $\widetilde h_v$ as before with $\widetilde u$ in place of $u$. In this way, the same argument as before using this function and the identity $\widetilde u\circ\psi_v=-\widetilde\lambda(v)P$, with $v\in V_0$, eventually yields $(\dd\widetilde\lambda)_{v_0}\equiv0$. Then Lemma \ref{lem:calculus} applies, thereby completing the proof.
\end{proof}

\begin{remark}
As essentially pointed out in \cite[pp. 1001--1002]{Bracci-Trans09},  Theorem \ref{thm:main1} and the proof of  \cite[Proposition 7.4]{Bracci-Trans09} imply that Proposition 7.5 there remains valid without the {\it a priori} $C^2$-regularity assumption on $u$, even for bounded strongly linearly convex domains with $C^{5,\,\alpha}$-boundary, where $0<\alpha<1$.
\end{remark}

\section{Proof of Theorem~\ref{thm:main2}}\label{sect:Pf-MainThm2}
Let $\mathbb B^{n+1}$ denote the unit ball in $\mathbb C^{n+1}$. Recall that its pluricomplex Poisson kernel with singularity at $e_{1}\coloneqq (1, 0,\ldots, 0)\in\partial\mathbb B^{n+1}$ is given by
  $$
  P_{\mathbb B^{n+1},\, e_1}(z)=-\frac{1-|z|^2}{|1-z_1|^2};
  $$
see, e.g., \cite[Section 1 and Theorem 5.1]{Bracci-Trans09}. To construct a family of plurisubharmonic functions on  $\mathbb B^{n+1}$ with the desired properties in Theorem~\ref{thm:main2}, we work on the Siegel upper
half-space
    $$
     \mathcal{U}^{n+1}=\Bigl\{(z,\, \zeta)\in\mathbb C^{n}\times\mathbb C\!: {\rm Im}\, \zeta-|z|^2>0 \Bigr\},
    $$
which is biholomorphic to $\mathbb B^{n+1}$ via the Cayley transform
   $$
   C(z,\, \zeta) =\bigg(\frac{2i z}{\zeta+i},\, \frac{\zeta-i}{\zeta+i}\bigg),     \quad (z,\,\zeta)\in \mathcal{U}^{n+1}.
   $$
An easy calculation shows
    $$
    P_{\mathbb B^{n+1},\, e_1}\circ C(z,\, \zeta)=\rho (z,\, \zeta) \coloneqq |z|^{2}-{\rm Im}\,\zeta.
    $$
The problem is then reduced to constructing a pairwise distinct family of continuous plurisubharmonic functions $\{v_{\lambda}\}_{\lambda\in(0,\, 1)}$
on $\mathcal{U}^{n+1}$ that solve the homogeneous complex Monge--Amp\`ere equation $(\ddc v)^{n+1}=0$, and are comparable with, but  not proportional to $\rho$.

We start with the following simple lemma, which is presumably well-known (compare with \cite[Chapter I]{Demaillybook}). We include a complete proof for the
reader's convenience.

\begin{lemma}\label{lem:convex-psh}
Let $f$ be a real-valued $C^{2}$ function on
   \[
    D=\big\{(x,\, y)\in \mathbb R^{2}\!: y>e^{x}\big\},
   \]
and  $v\!:\mathcal{U}^{n+1} \!\setminus\! (\{0\}\times\mathbb C)\rightarrow \mathbb R$ be the function given by
   \[
    v(z,\, \zeta)=f\big(2\log|z|,\, {\rm Im}\, \zeta \big).
   \]
If $f$ is convex and $\frac{\partial f}{\partial x}\geq 0$, then $v$ is plurisubharmonic, and the determinant of its complex Hessian is
related to that of the real Hessian of $f$ by the following formula:
\begin{equation*}\label{eq:complex-real-hessian}
    \det {\rm Hess}_{\mathbb C} v(z,\, \zeta)
      =\frac{1}{4|z|^2} \bigg(\frac{1}{|z|^2} \frac{\partial f}{\partial x}  \bigl(2\log|z|,\, {\rm Im}\,\zeta\bigr) \bigg)^{n-1}
         \det {\rm Hess}_{\mathbb R} f \bigl(2\log|z|,\, {\rm Im}\,\zeta\bigr).
\end{equation*}
\end{lemma}

\begin{proof}
By unitary invariance, it suffices to verify the desired conclusion at points of the form  $(re_{1},\, \zeta)\in \mathcal{U}^{n+1}$, where $e_{1}\coloneqq (1, 0,\ldots, 0)\in\mathbb C^{n}$ and $r>0$. At each such point, the complex Hessian of $v$ is equal to
\begin{equation*}\renewcommand{\arraystretch}{2.0}
    \begin{pmatrix}
      \dfrac{1}{r^{2}}f_{xx}  & 0 & \dfrac{i}{2r}f_{xy} \\
      0 & \dfrac{1}{r^{2}}f_x I_{n-1} & 0 \\
      -\dfrac{i}{2r}f_{xy} & 0 & \dfrac14 f_{yy}
    \end{pmatrix}(2\log r,\, {\rm Im}\,\zeta),
\end{equation*}
where $I_{n-1}$ denotes the $(n-1)\times (n-1)$ identity matrix, and $f_x$, $f_{xy}$ denote the partial derivatives of $f$ with respect to  $x$ and $xy$, respectively, and so on.
The convexity of $f$ and the condition $f_x\geq 0$ imply that the preceding
matrix is positive semidefinite. Evaluating its determinant via a Laplace expansion along the first row gives
\begin{align*}
\det{\rm Hess}_{\mathbb C} v(re_{1},\, \zeta)
  &=\left(\frac{1}{r^{2}}f_x(2\log r,\, {\rm Im}\,\zeta)\right)^{n-1} \frac{1}{4r^{2}}\bigl(f_{xx}f_{yy}-f_{xy}^{2}\bigr)(2\log r,\, {\rm Im}\,\zeta)\\
  &=\frac{1}{4r^{2}}\left(\frac{1}{r^{2}}f_x(2\log r,\, {\rm Im}\,\zeta)\right)^{n-1}
    \det{\rm Hess}_{\mathbb R} f(2\log r,\, {\rm Im}\,\zeta).
\end{align*}
This completes the proof.
\end{proof}

To motivate the following construction we observe that the convex function
$e^{x}-y$, restricted to $D$ (given as in Lemma \ref{lem:convex-psh}), can be expressed as the envelope of a
one-parameter family of affine functions:
  $$
  e^x-y=\sup_{t\in\mathbb R} c(t)\bigl(e^{t}(1+x-t)-y\bigr),  \quad {\rm with }\  \,  c(t)\equiv 1.
  $$
We now replace the constant function $c(t)\equiv 1$ by a general positive,
bounded $C^{\infty}$ function $c_{\lambda}(t)$ and define
  \begin{equation}\label{eq:f-lambda-def}
   f_{\lambda}(x,y)\coloneqq \sup_{t\in\mathbb R} c_{\lambda}(t)\bigl(e^{t}(1+x-t)-y\bigr), \quad (x,\, y)\in D.
  \end{equation}
With our goal in mind and also taking certain practical manipulations into
account, we are led to naturally choose
  \begin{equation*}\label{eq:c-lambda-def}
    c_{\lambda}(t) \coloneqq \exp\left(-\lambda\frac{e^{t}}{1+e^{t}}\right),     \quad 0<\lambda<1.
  \end{equation*}
It is evident that each $c_{\lambda}$ is decreasing, with
   $$
    \lim_{t\to-\infty}c_{\lambda}(t)=1 \qquad {\rm and } \qquad   \lim_{t\to+\infty}c_{\lambda}(t)=e^{-\lambda}.
   $$

We now investigate the properties of the family $\{f_{\lambda}\}_{\lambda\in(0,\, 1)}$.

\begin{lemma}\label{lem:global-estimates}
Each $f_{\lambda}$ is convex on $D$ and is increasing in its first variable. Moreover, it satisfies the following two-sided estimate:
  $$
  e^{x}-y \leq f_{\lambda}(x,\, y) \leq e^{-\lambda}(e^{x}-y),   \quad (x,\, y)\in D.
  $$
\end{lemma}

\begin{proof}
Being the supremum of a family of (negative) affine functions, $f_{\lambda}$ is convex on $D$. It is also evident that $f_{\lambda}$ is increasing in its first variable.

As for the estimate for  $f_{\lambda}$, the convexity of the exponential gives
   $$
    e^t(1+x-t)\leq e^x,  \quad x,\, t\in\mathbb R.
   $$
Since $e^x-y<0$ on $D$ and $c_{\lambda}(t)\geq e^{-\lambda}$ on $\mathbb R$, each member of the family defining $f_{\lambda}$ is at most $e^{-\lambda}(e^x-y)$. Hence
   $$
   f_{\lambda}(x,\, y) \leq e^{-\lambda}(e^{x}-y) \quad {\rm on }\ \, D.
   $$
On the other hand, choosing $t=x$ on the right-hand side of \eqref{eq:f-lambda-def} yields
   $$
   f_{\lambda}(x,\, y) \geq c_{\lambda}(x)(e^x-y) \geq e^x-y   \quad {\rm on }\ \, D.
   $$
The proof is complete.
\end{proof}

The next lemma establishes the smoothness of the family $\{f_{\lambda}\}_{\lambda\in(0,\, 1)}$ and the rank-one structure of their real Hessians.

\begin{lemma}\label{lem:unique-maximizer} For every $\lambda\in(0,\, 1)$ and $(x,\, y)\in D$, the supremum defining $f_{\lambda}(x,\,y)$ is attained at a unique point $t=t_{\lambda}(x,\,y)$. Moreover, this maximizer depends smoothly on $(x,\, y)$  so that $f_{\lambda}$ is
smooth, with
   $$
   {\rm grad}\, f_{\lambda}=c_{\lambda}\circ t_{\lambda} \bigl(e^{t_{\lambda}},\, -1\bigr)
   \qquad {\rm and } \qquad  \rank {\rm Hess}_{\mathbb R}f_{\lambda}\equiv 1.
   $$
\end{lemma}

\begin{proof}
Given $\lambda\in(0,\, 1)$ and $(x,\, y)\in D$, we write $c=c_{\lambda}$ and
   $$
   E(t)\coloneqq c(t)\bigl(e^{t}(1+x-t)-y\bigr).
   $$
A straightforward calculation shows
   $$
   F(t)\coloneqq \frac{E'(t)}{e^{t}c(t)}
   =\bigl(1-\kappa(t)\bigr)(x-t)-\kappa(t)\bigl(1-ye^{-t}\bigr),
   $$
where
   $$
   \kappa(t)\coloneqq -\frac{c'(t)}{c(t)}=\frac{\lambda e^{t}}{(1+e^{t})^{2}}.
   $$
Since $0<\kappa(t)\leq \lambda/4 <1/4$, it follows that
 \begin{equation}\label{eq:F-limits}
    \lim_{t\to-\infty}F(t)=+\infty  \qquad {\rm and } \qquad \lim_{t\to+\infty}F(t)=-\infty.
 \end{equation}

We next show that $F$ admits exactly one zero point. Once this is proved, it follows immediately that $E$ attains its global strict maximum at this point,
as desired. To this end, first observe that at each zero point of $F$,
  \begin{equation}\label{eq:critical-identity}
    x-t = \frac{\kappa(t)}{1-\kappa(t)}\bigl(1-ye^{-t}\bigr).
  \end{equation}
From this and the fact that $y>e^{x}$, it follows that $ye^{-t}>1$. (Otherwise $1-ye^{-t}\geq0$, and the preceding equality would give $x\geq t$, hence $ye^{-t}>e^{x-t}\geq1$, a contradiction.) Note also that
   $$
    \kappa'(t)-\kappa(t)\bigl(1-\kappa(t)\bigr)
    =
    -\frac{\lambda e^{2t}(2+2e^{t}-\lambda)}{(1+e^{t})^{4}}
    <0.
   $$
Now differentiating $F$ and using this fact together with \eqref{eq:critical-identity}, we conclude that at each zero point of $F$,
  $$
  F'(t)=\left(\frac{\kappa'(t)}{1-\kappa(t)}-\kappa(t) \right)\bigl(ye^{-t}-1\bigr)-1 <-1.
  $$
It follows from this and \eqref{eq:F-limits} that $F$ admits exactly one zero point, which we denote by $t=t_{\lambda}(x,\,y)$,  and since $F'$ does not vanish there, the resulting
function $D\ni (x,\,y)\mapsto t_{\lambda}(x,\,y)$ is smooth by the implicit function theorem.

The envelope theorem now gives
  $$
  {\rm grad}\, f_{\lambda}=c_{\lambda}\circ t_{\lambda} \bigl(e^{t_{\lambda}},\, -1\bigr).
  $$
Differentiating this identity with respect to $x$ and $y$, respectively, yields
\begin{equation*}\label{eq:hessian-f}\renewcommand{\arraystretch}{1.8}
 {\rm Hess}_{\mathbb R} f_{\lambda}=\frac{\partial t_{\lambda}}{\partial x}
   \begin{pmatrix}
    e^{t_{\lambda}} (c_{\lambda}+c_{\lambda}')\circ t_{\lambda}  & -c_{\lambda}'\circ t_{\lambda}\\
     -c_{\lambda}'\circ t_{\lambda}  &  \dfrac{e^{-t_{\lambda}}(c_{\lambda}'\circ t_{\lambda})^{2}}
    {(c_{\lambda}+c_{\lambda}')\circ t_{\lambda} }
\end{pmatrix}.
\end{equation*}
(Note that $c_{\lambda}(t)+c_{\lambda}'(t)\neq0$ for all $t\in\mathbb R$, as is easily verified.) Similarly, differentiating  \eqref{eq:critical-identity}---with $t_{\lambda}$ in place of $t$---with respect to $x$ yields
  $$
  \frac{\partial t_{\lambda}}{\partial x}\neq0 \quad {\rm on }\ \, D.
  $$
Hence $\rank {\rm Hess}_{\mathbb R} f_{\lambda}\equiv 1$, and the proof is complete.
\end{proof}

\begin{remark}
As indicated by the preceding proof, Lemma~\ref{lem:unique-maximizer} remains valid when $\lambda$ is allowed to take values in the broader range $(0,\,2]$.
\end{remark}

\begin{lemma}\label{lem:asymptotics}
 Each $f_{\lambda}$ exhibits the following asymptotic behavior:
   \begin{enumerate}[label=\textup{(\roman*)}, leftmargin=2.0pc, parsep=4pt]
     \item  $f_{\lambda}(x,\, y)$ converges locally uniformly to $-y$ on  $(0,\, \infty)$ as $x\to-\infty$;

     \item  $f_{\lambda}(x,\, y)/y$ converges locally uniformly to $-e^{-\lambda}$  on $\mathbb R$ as $y\to+\infty$.
   \end{enumerate}
\end{lemma}

\begin{proof}
We first establish the relatively easy part, namely part (ii). By Lemma \ref{lem:global-estimates},
  $$
   \frac{f_{\lambda}(x,\, y)}{-y}\geq e^{-\lambda}\bigg(1-\frac{e^x}{y}\bigg)  \quad {\rm on }\ \, D.
  $$
On the other hand, choosing $t=\frac12\log y$ on the right-hand side of \eqref{eq:f-lambda-def} yields the upper bound estimate
  $$
  \frac{f_{\lambda}(x,\, y)}{-y}\leq c_{\lambda}\big(\frac12\log y\big)\bigg(1-\frac{1+x-\frac12\log y}{\sqrt y}\bigg)  \quad {\rm on }\ \, D.
  $$
The desired result then follows immediately by combining these two estimates with the fact that $\lim_{t\to+\infty}c_{\lambda}(t)=e^{-\lambda}$.

Next we prove part (i). Let $t=t(x,\,y)$ be the function determined in Lemma~\ref{lem:unique-maximizer}, and retain the notation
$\kappa=-c'/c$ introduced in its proof. Then the identity
  $$
  x-t(x,\,y)= \frac{\kappa\circ t(x,\,y) }{1-\kappa\circ t(x,\, y) }
    \bigl(1-ye^{-t(x,\, y)}\bigr)
  $$
holds on $D$. Note that as $x\to-\infty$, the function $t(x,\, y)$ tends to $-\infty$ locally uniformly in $y\in(0,\,\infty)$. If not, along some sequence with $y$ in a
fixed compact subset of $(0,\,\infty)$, the values of $t(x, \,y)$ would be bounded below. Then the left-hand side of the preceding identity would tend to
$-\infty$, while the right-hand side would remain bounded below, yielding a contradiction.

Now
   $$
   \lim_{x\to-\infty}\kappa\circ t(x,\, y)=0  \qquad {\rm and } \qquad
   \lim_{x\to-\infty}e^{-t(x,\, y)}\kappa\circ t(x,\, y)=\lambda.
   $$
In view of the preceding identity, this in turn implies
   $$
    \lim_{x\to-\infty} x-t(x,\, y) = -\lambda y.
   $$
In particular, $x-t(x,\, y)$ remains locally bounded. Evaluating the maximizing
member of the envelope in \eqref{eq:f-lambda-def} and using the fact that $\lim_{t\to -\infty}c_{\lambda}(t)=1$, we then conclude that
   $$
    f_{\lambda}(x,\, y)=c\circ t(x,\, y) \Big( e^{t(x,\, y)}\bigl(1+x- t(x,\, y)\bigr)-y \Big)
    \rightarrow  -y,
   $$
as desired. It is also easy to see that this convergence is uniform whenever $y$ stays in a compact subset of $(0,\, \infty)$.
\end{proof}

With the family $\{f_{\lambda}\}_{\lambda\in(0,\, 1)}$ constructed above at hand, we define
\begin{equation*}\label{eq:v-lambda-def}\renewcommand{\arraystretch}{1.8}
  v_{\lambda}(z,\, \zeta)\coloneqq
    \begin{cases}
       f_{\lambda}\bigl(\log|z|^{2},\, {\rm Im}\,\zeta\bigr), & \quad (z,\, \zeta)\in \mathcal{U}^{n+1}\!\setminus\!(\{0\}\times\mathbb C);\\
       \qquad -{\rm Im}\,\zeta, & \quad (0,\, \zeta)\in \mathcal{U}^{n+1}.
     \end{cases}
\end{equation*}

We are now in a position to state and prove the following theorem, which is the main result of this section and a reformulation of Theorem~\ref{thm:main2} on the Siegel upper half-space $\mathcal{U}^{n+1}$.

\begin{theorem}\label{thm:upper-half-space-family}
The family $\{v_{\lambda}\}_{\lambda\in(0,\, 1)}$ enjoys the following properties:
  \begin{enumerate}[label=\textup{(\roman*)}, leftmargin=2.0pc, parsep=4pt]
    \item Each $v_{\lambda}$ is continuous, plurisubharmonic on $\mathcal{U}^{n+1}$ and solves the homogeneous complex Monge--Amp\`ere equation
          $(\ddc v_{\lambda})^{n+1}=0$;

    \item Each $v_{\lambda}$ is comparable with $|z|^{2}-{\rm Im}\,\zeta$:
          $$
          |z|^{2}-{\rm Im}\,\zeta \leq  v_{\lambda}(z,\, \zeta) \leq e^{-\lambda} \bigl(|z|^{2}-{\rm Im}\,\zeta\bigr)
          \quad {\rm on }\ \, \mathcal{U}^{n+1},
          $$
          but is not proportional to $|z|^{2}-{\rm Im}\,\zeta$;
    \item $v_{\lambda}\neq v_{\mu}$ whenever $\lambda\neq\mu$.
   \end{enumerate}
\end{theorem}

\begin{proof}
We first show that each $v_{\lambda}$ is continuous and plurisubharmonic  on $\mathcal{U}^{n+1}$. By Lemmas~\ref{lem:global-estimates} and \ref{lem:unique-maximizer}, each $f_{\lambda}$ is a $C^{\infty}$ convex function on $D$ that is increasing in its first variable. Lemma \ref{lem:convex-psh} therefore implies that each $v_{\lambda}$ is smooth and plurisubharmonic on $\mathcal{U}^{n+1} \!\setminus\! (\{0\}\times\mathbb C)$.  Together with Lemma~\ref{lem:asymptotics}(i),  this shows that $v_{\lambda}$ is continuous throughout $\mathcal{U}^{n+1}$. Moreover, $v_{\lambda}$ is plurisubharmonic on the whole $\mathcal{U}^{n+1}$, as the submean value property along complex lines follows readily from its construction and its plurisubharmonicity away from $\{0\}\times\mathbb C$.  Alternatively, one may deduce the desired plurisubharmonicity from the classical removable singularity theorem for plurisubharmonic functions (see, e.g., \cite[Chapter I, Theorem 5.24]{Demaillybook}), since $v_\lambda$ is locally bounded above near $\mathcal{U}^{n+1}\cap(\{0\}\times\mathbb C)$ and the
resulting plurisubharmonic extension agrees with $v_\lambda$ by
continuity.

We next show that
   $$
    (\ddc v_{\lambda})^{n+1}=0 \quad {\rm on }\ \, \mathcal{U}^{n+1}.
   $$
Since $v_{\lambda}$ is locally bounded on $\mathcal{U}^{n+1}$, its complex Monge--Amp\`ere measure $(\ddc v_{\lambda})^{n+1}$ does not charge pluripolar Borel subsets of $\mathcal{U}^{n+1}$; see, e.g., \cite{Bedford82, GZ-MAeqnbook}. In particular, it assigns zero mass to $\mathcal{U}^{n+1}\cap(\{0\}\times\mathbb C)$. It therefore suffices to show that $(\ddc v_{\lambda})^{n+1}=0$ on $\mathcal{U}^{n+1} \!\setminus\! (\{0\}\times\mathbb C)$. This follows immediately from Lemmas~\ref{lem:convex-psh} and \ref{lem:unique-maximizer}.

In addition, Lemma~\ref{lem:global-estimates} together with continuity shows that each $v_{\lambda}$ is comparable with the function
    $$
    \rho(z,\, \zeta)=|z|^{2}-{\rm Im}\,\zeta.
    $$
Now what remains is to show that each $v_{\lambda}$ is not proportional to $\rho$ and that $v_{\lambda}\neq v_{\mu}$ whenever $\lambda\neq\mu$. For this, it suffices to note that
   $$
    v_{\lambda}=\rho \quad {\rm on }\ \, \mathcal{U}^{n+1}\cap(\{0\}\times\mathbb C)
   $$
and, by Lemma~\ref{lem:asymptotics}\textup{(ii)},
   $$
   \lim_{\mathbb R\ni y\to+\infty} \frac{v_{\lambda}(z,\, iy)}{\rho(z,\,iy)}=e^{-\lambda} \quad {\rm for}\ \,  z\in\mathbb C^{n} \!\setminus\! \{0\}.
   $$
\end{proof}

\end{document}